\documentclass[12pt,leqno,twoside,a4paper]{article}
\usepackage{amsfonts,amsmath,amsthm,amssymb}
\usepackage[english]{babel}
\newtheorem{theorem}{Theorem}[section]
\newtheorem{prop}[theorem]{Proposition}
\newtheorem{lemma}[theorem]{Lemma}
\newtheorem{corollary}[theorem]{Corollary}
\theoremstyle{definition}
\newtheorem{definition}[theorem]{Definition}
\newtheorem{example}[theorem]{Example}
\theoremstyle{remark}
\newtheorem{remark}[theorem]{Remark}

 \numberwithin{equation}{section}

\newcommand{\de}{\delta}

\newcommand{\T}[1]{{#1}[x;\si ,\de]}
\newcommand{\Tn}{R[x;\si ,\de]}

\newcommand{\A}[1]{{#1}[x;\si]}
\newcommand{\si}{\sigma}

\newcommand{\udim}[2]{\mbox{\rm udim}_{#1}({#2})}

\begin{document}

\title{\normalsize \sc  Ore Extensions Satisfying a Polynomial Identity}
\author{ {\bf \normalsize Andr\'{e} Leroy and Jerzy
Matczuk\footnote{ The research was supported by Polish KBN grant
No. 1 P03A 032 27 }} \vspace{6pt}\\
 \normalsize  Universit\'{e} d'Artois,  Facult\'{e} Jean Perrin\\
\normalsize Rue Jean Souvraz,  62 307 Lens, France\\
   \normalsize  E-mail: leroy@euler.univ-artois.fr \vspace{6pt}\\
 \normalsize Institute of Mathematics, Warsaw University,\\
\normalsize Banacha 2, 02-097 Warsaw, Poland\\
 \normalsize E-mail: jmatczuk@mimuw.edu.pl}
\date{ }
\maketitle\markboth{\rm  A. LEROY AND J.MATCZUK}{ \rm PI ORE
EXTENSIONS}


\begin{abstract} Necessary and sufficient
 conditions for an Ore extension
 $S=R[x;\si,\de]$ to be a \mbox{\rm PI} ring are given in the case $\si$ is an injective
 endomorphism of a  semiprime
 ring $R$ satisfying the \mbox{\rm ACC} on annihilators. Also, for an arbitrary endomorphism
 $\tau$ of $R$, a characterization of Ore extensions $R[x;\tau]$
 which are \mbox{\rm PI} rings is given, provided the coefficient ring $R$ is noetherian.

\end{abstract}

\section*{\sc \normalsize Introduction}

The aim of the paper is to give necessary and sufficient
conditions for an Ore extension $R[x;\si,\de]$ to satisfy a
polynomial identity.  One of the special feature is that we do not
assume that $\si$ is an automorphism.

Clearly if $\Tn$ satisfies a polynomial identity, then $R$ has to
be a \mbox{\rm PI} ring as well.  Henceforth we  will always
assume that $R$ is a \mbox{\rm PI} ring.

In \cite{PV}, Pascaud and Valette showed that when $R$ is
semiprime  and $\si$ is an automorphism of $R$, then the Ore
extension $R[x;\si]$ satisfies a polynomial identity if and only
if $\si$ is of finite order on the center of $R$. We shall obtain
similar results even when $\si$ is not an automorphism.

 The \mbox{\rm PI} property of
general Ore extensions $\T{R}$ was studied by Cauchon in his
thesis \cite{C} in the case when the base ring $R$ is simple and
$\si$ is an automorphism of $R$.  In particular, Cauchon remarked
that nonconstant central polynomials of $\T{R}$ appears naturally
in this context.

On the other hand, the case of a noetherian base ring  was
considered by Damiano and Shapiro in \cite{DS}. They proved, in
particular, that the Ore extension $\A{R}$ over a noetherian PI
ring $R$ satisfies a polynomial identity if and only if the
automorphism $\si$ is of finite order on the center of $R/B$,
where $B$ denotes the prime radical of $R$.  This result will be
generalized to arbitrary endomorphism of $R$ in the last section.

 Let us also mention that, mainly
for Ore extensions coming from quantum groups, some authors are
also interested in the relations between the \mbox{\rm PI} degree
of a ring $R$ and the \mbox{\rm PI} degree of Ore extensions over
$R$ (Cf. e.g. \cite{BGO},\cite{Jon}).

 A somewhat related work is  that of  Bergen
and Grzeszczuk (Cf.  \cite{BG2}), where a  characterization of
 smash products $R\# U(L)$ satisfying polynomial identity is given, where
 R is a semiprime algebra of characteristic 0 acted by a Lie color algebra $L$.

In Section $1$, we recall some classical results and develop tools
that will play an essential role in later sections.

In Section $2$, we analyse the special case when $R$ is a prime
ring. The main results being Theorem \ref{main prime}, Theorem
\ref{prime case} and Theorem \ref{prime caracterization with
udim}.   We prove, in particular, that for a prime \mbox{\rm PI}
ring $R$ and an injective endomorphism $\si$ of $R$, $\T{R}$ is
\mbox{\rm PI} if and only if there exists a nonconstant polynomial
in the center of $\T{R}$ with regular leading coefficient if and
only if the center $Z(R)$ of $R$ has finite uniform dimension over
$Z(R)^{\si,\de}:=\{z \in Z(R) \, | \, \si(z)=z,\, \de(z)=0 \}$.

In Section $3$, we extend the results from the previous section to
the case of a semiprime coefficient ring $R$ satisfying the ACC on
annihilators.  We first recall the results of Cauchon and Robson
related to the action of an injective endomorphism $\si$ and a
$\si$-derivation $\de$ on a semisimple ring . Then we study the
case of  Ore extensions of endomorphism type ($\de=0$) showing in
particular (Proposition \ref{semiprime endomorhism})   that the
above mentioned result of Pascaud and Valette can be generalized
to the case when $\si$ is an  injective endomorphism of $R$.  The
general Ore extension $\T{R}$ is then analysed, first in the case
of a semisimple base ring  (Cf. Corollary \ref{C semisimple}).
Then  it is shown, in  Theorem \ref{main semiprime}, that the
necessary and sufficient conditions for $\T{R}$ to be a \mbox{\rm
PI} ring obtained in Section $2$ are also valid under the
assumption that $R$ is a semiprime \mbox{\rm PI} ring with the
\mbox{\rm ACC} on annihilators.

The last section is devoted to the study of the \mbox{\rm PI}
property of the Ore extension $\A{R}$ where $\si$ is an arbitrary
endomorphism of $R$ and the  coefficient ring $R$ is noetherian.
In particular, we give in Theorem \ref{general case}, a necessary
and sufficient conditionh for $\A{R}$ to satisfy a polynomial
identity.

\section{\sc \normalsize Preliminaries}

Throughout the paper $R$ stands for an associative ring with unity
and $Z(R)$ for its center.  For any multiplicatively closed subset
$S$ of $Z(R)$, $RS^{-1}$ denotes the localization of $R$ with
respect to the set of all regular elements from $S$. In
particular,  $RZ(R)^{-1}$ is the localization of $R$ with respect
to the Ore set of all central regular elements of $R$.

 For a right $R$-module $M$, $\udim{R}{M}$
denotes its uniform dimension.

In the following proposition we gather classical results which are
consequences of a generalized Posner's Theorem and the theorem of
Kaplansky (Cf. Rowen's book \cite{R}).
\begin{prop}\label{classics}
 For a semiprime \mbox{\rm PI} ring $R$, the following conditions are equivalent:
\begin{enumerate}
  \item $R$ is a left (right) Goldie ring;
  \item $R$ satisfies the \mbox{\rm ACC} condition on left (right)
  annihilators;
  \item $R$ has finitely many minimal prime ideals;
  \item $R$ possesses a semisimple classical left (right) quotient ring $Q(R)$ which is equal to
  the central localization $RZ(R)^{-1}$.
  \end{enumerate}
If $R$ satisfies one of the above equivalent conditions and
$\bigoplus_{i=1}^n B_i$ is a decomposition of the semisimple ring
$Q(R)$ into simple components, then $\dim_{Z(B_i)}B_i$ is finite,
for every $1\leq i\leq n$. In particular, $\udim{Z(Q(R))}{Q(R)}<
\infty$.
\end{prop}

We will use frequently the above proposition without referring to
it.

The following observation is probably well-known but we could not
find it in the literature:
\begin{prop}
 Let $B=\bigcup_{i=0}^{\infty}A_i$ be a filtered
ring and $\mbox{\rm gr}(B)=\bigoplus_{i=0}^{\infty}A_i/A_{i-1}$
denote its associated graded ring.  If $B$ satisfies a polynomial
identity then $\mbox{\rm gr}(B)$ also satisfies a polynomial
identity.
\end{prop}
\begin{proof}
Let $G(B)$ denote the Rees ring of $B$, that is $G(B)$ is a
subring of the polynomial ring $B[x]$ consisting of all
polynomials $\sum_i a_ix^i\in B[x]$ such that $a_i\in A_i$, for
any $i\geq 0$. Then $G(B)$ is a \mbox{\rm PI} ring as a subring of
the \mbox{\rm PI} ring $B[x]$. It is known that the ring
$G(B)/xG(B)$ is isomorphic to $\mbox{\rm gr}(B)$ and the thesis
follows.
\end{proof}

 An Ore extension of a ring $R$ is denoted by $\Tn$,
 where $\si$ is an endomorphism of $R$ and $\de$ is a $\si$-derivation,
 i.e. $\de \colon R\rightarrow R$ is an additive map such that
$\de(ab)=\si(a)\de(b)+\de(a)b$, for all $a,b\in R$.
 Recall that
elements of $\Tn$ are polynomials in $x$ with coefficients written
on the left. Multiplication in $\Tn$ is given by the
multiplication in $R$ and the condition $xa=\si(a)x+\de(a)$, for
all  $a\in R$.

 The Ore extension $R[x;\si,\de]$ has a natural
filtration given by the degree and the associated graded ring is
isomorphic to $R[x;\si]$. Therefore, by the above proposition, we
have:
\begin{corollary}\label{graded PI}
If $\T{R}$ is a \mbox{\rm PI} ring, then $R[x;\si]$  also
satisfies a polynomial identity.
\end{corollary}

\begin{lemma}
\label{generalities on the center} Let $\si$ be an endomorphism of
$R$. Then:
\begin{enumerate}
\item Suppose that  $\si$ is injective and there
 exists $n\ge 1$ such that $\si^n\vert_{Z(R)}$
is an automorphism of $Z(R)$. Then $\si\vert_{Z(R)}$ is an
automorphism of $Z(R)$.
\item Suppose that $R$ is a semiprime \mbox{\rm PI} ring and $\si$ is injective when
restricted to the center $Z(R)$, then $\si$ is injective on $R$.
\item Suppose that $R$ is simple finite dimensional over $Z(R)$ and $\si\vert_{Z(R)}$
is an automorphism of $Z(R)$ then $\si$ is an automorphism of $R$.
\end{enumerate}
\end{lemma}
\begin{proof}
$(1)$. Let  $[a,b]$ denote  the commutator of elements $a,b\in R$
, i.e. $[a,b]=ab-ba$. Pick $n\geq 1$  such that
$\si^n\vert_{Z(R)}$ is an automorphism of $Z(R)$.

Then, for any $r\in R$ and $z\in Z(R)$,  we have
$$\si^{n-1}([\si(z),r])=[\si^n(z),\si^{n-1}(r)]=0.$$  Hence
$[\si(z),r]\in \ker \si^{n-1}=0$. This gives $(1)$.

The statement $(2)$ is clear, as any nonzero ideal of a semiprime
\mbox{\rm PI} ring intersects the center nontrivially.

$(3).$  Since $R$ is simple, $Z(R)$ is a field.  Let $R'$ be the
left $Z(R)$-linear space $R$ with the action of $Z(R)$
 twisted by $\si$, i.e. $z\cdot r=\si(z)r$, for $z\in Z(R)$, $r\in
 R'$. Now the thesis is a consequence of  the fact
 that $\si \colon _{Z(R)}R\rightarrow  _{Z(R)}R'$ is an
 injective homomorphism of $Z(R)$-linear spaces of the same finite dimension.
\end{proof}

Braun and  Hajarnavis showed in \cite{BH} that if $\si$ is an
injective endomorphism of a prime  noetherian  \mbox{\rm PI} ring
$R$ such that $\si|_{Z(R)}=\mbox{id}_{Z(R)}$, then $\si$ is an
automorphism of $R$. We will see in Example \ref{example}, that if
$R$ is a prime \mbox{\rm PI} ring (so it has the \mbox{\rm ACC} on
annihilators), then $\si$ does not have to be onto if
$\si|_{Z(R)}=\mbox{id}_{Z(R)}$.

In the lemma below we quote some known  results. The first
statement is a special case of a result of Jategaonkar (Cf.
Proposition 2.4 \cite{J}). The second one is exactly Theorem 3.8
and Proposition 3.6 from \cite{LM1}, respectively.
\begin{lemma} \label{known facts}
Let $R$ be a semiprime left Goldie ring. Suppose that the
endomorphism $\si$ is injective. Then:
\begin{enumerate}
   \item Let $\mathcal{C}$ be the
  set of all regular elements of $R$. Then $\si(\mathcal{C} )\subseteq
  \mathcal{C}$;
  \item  $\T{R}$  is a  semiprime left Goldie ring. When $R$ is prime, then $\T{R}$
  is also  a  prime ring.
\end{enumerate}
\end{lemma}

 As we will see in the
following result, the above lemma  will enable us to reduce some
of our considerations to the case when the coefficient ring $R$ is
semisimple.
\begin{prop}\label{reduction to semisimple}
 Let $R$ be  a semiprime left Goldie ring.
 Suppose that $\si$ is  an injective endomorphism of $R$.  Then $\si$ and $\de$
 can be uniquely extended to the classical ring of quotient $Q(R)$  of $R$, and the
 following conditions are equivalent:
\begin{enumerate}
  \item $\T{R}$ is a \mbox{\rm PI} ring.
  \item $\T{Q(R)}$ is a \mbox{\rm PI} ring.
\end{enumerate}
\end{prop}
\begin{proof}By Lemma \ref{known facts}(1), $\si(\mathcal{C})\subseteq \mathcal{C}$,
 where $\mathcal{C}$ is the set of all regular elements of $R$.
This means that $\si$ and $\de$ can be uniquely extended to an
injective endomorphism $\si$ and a $\si$-derivation $\de$ of
$Q(R)=\mathcal{C}^{-1}R$ and we can consider the Ore extension
$\T{Q(R)}$.

The implication $(2)\Rightarrow (1) $ is obvious.

$(1)\Rightarrow (2).$
  Suppose $\T{R}$
 satisfies a polynomial identity. Then, by Lemma  \ref{known facts}(2),
  $\T{R}$ is a semiprime \mbox{\rm PI} left Goldie
 ring. Thus,  $Q(\T{R})$ is  a  semisimple \mbox{\rm PI} ring.

Clearly all elements from the set  $\mathcal{C}$   are invertible
in $Q(R)[x;\si,\de]$ and every element from $Q(R)[x;\si ,\de]$ can
be presented in the form $c^{-1}p$ for some $c\in \mathcal{C}$ and
$p\in R[x;\si ,\de]$. This means that $\mathcal{C}$ is a left Ore
set in $R[x;\si ,\de]$ and $\mathcal{S}^{-1}(R[x;\si
,\de])=Q(R)[x;\si,\de]$.
 This yields, in particular, that there exists a natural embedding of
$\T{Q(R)}$ into the \mbox{\rm PI} ring $Q(\T{R})$. This shows that
$\T{Q(R)}$ satisfies a polynomial identity.
\end{proof}

 In the sequel we will need
the following:

\begin{lemma}
\label{udim and localizations} Let $R$ be a ring and $\cal S$ be a
right Ore set of regular elements.  If $M$ is a right $R$-module
which is $\cal S$-torsion free then
$$ \udim{R}{M} = \udim{{R{\cal S}^{-1}}}{M\otimes_R R{\cal S}^{-1}} \,.
$$
In particular, if $R$ is a commutative integral domain and $M$ is
a torsion free right $R$-module then $\udim{R}{M} =
\dim_K(M\otimes_R K)$, where $K$ is the field of fractions of $R$.
\end{lemma}
\begin{proof}
The particular case comes from \cite{L} Theorem 6.14 and the proof
given there can be easily extended to get the first statement.
\end{proof}

 In the next lemma we will consider Ore extensions of the form $R[x;\phi]$, where  $\phi$
denotes either an automorphism or a derivation of $R$. $R^{\phi}$
will denote a subring of constants, i.e.  $R^\phi=\{x \in R \,
\vert \, \phi(x)=x\}$ when $\phi$ is an automorphism and
$R^\phi=\{x \in R \, \vert \, \phi(x)=0\}$, when $\phi$ is a
derivation.
\begin{lemma}
\label{the untwisted cases} Let $R$ be a ring and
$Z=Z(R[x;\phi])$, where $\phi$ is either an automorphism or a
derivation of $R$. Then:
\begin{enumerate}
\item $\udim{R^\phi}{R} \leq \udim{Z}{R[x;\phi]}$.
\item Suppose that $R$ is a semiprime ring with the \mbox{\rm ACC} on
annihilators. If $R[x;\phi]$ is a \mbox{\rm PI} ring, then:
\begin{enumerate}
\item $\udim{R^\phi}{R} <\infty$.
\item  If every regular element of $Z(R)^{\phi}$ is regular in $Z(R)$,
then:\\
 $Z(R)Z(R)^{-1}=Z(R)(Z(R)^{\phi})^{-1}$ and $Q(R)=R(Z(R)^{\phi})^{-1}$.
\end{enumerate}
\end{enumerate}
\end{lemma}
\begin{proof} $(1)$.  For $f=\sum_{i=0}^na_ix^i\in Z$, let $\phi(f)$
denote $\sum_{i=0}^n\phi(a_i)x^i$.  Then, $0=[x,f]=(\phi(f)-f)x$,
if $\phi$ is an
 automorphism of $R$. If $\phi$ is a
 derivation of $R$, then $0=[ x,f]=\phi(f)$. The above yields that  $Z
\subseteq R^{\phi}[x]$.

For any element $r\in R^{\phi}$ we have $xr=rx$. Therefore, if $M$
is a right $R^\phi$-submodule of $R$, then $M[x]$ is a right
$R^\phi[x]$-submodule of $R[x;\phi]$. In particular, $M[x]$ has
also a structure of $Z$-module, as $Z\subseteq R^\phi[x]$. One can
easily  check that direct sums of $R^\phi$-submodules of $R$ lift
to  direct sums of  $R^\phi[x]$-submodules of $R[x;\phi]$.
Therefore  $\udim{R^{\phi}}{R}\leq \udim{R^{\phi}[x]}{R[x;\phi]}
\leq \udim{Z}{R[x;\phi]}$.

$(2)(a)$. Suppose that $R[x,\phi]$ is a \mbox{\rm PI} ring and the
coefficient ring $R$ is semiprime with the \mbox{\rm ACC}
condition on annihilators. Then $R$ also satisfies a polynomial
identity, so $R$ is a semiprime \mbox{\rm PI} left Goldie ring.
Therefore, by Lemma \ref{known facts}, $R[x,\phi]$ is a semiprime
Goldie \mbox{\rm PI} ring with a semisimple quotient ring
$Q(R[x,\phi])=R[x,\phi]Z^{-1}$.

 Making use of  the statement (1), Lemma \ref{udim and localizations}  and
 Proposition \ref{classics}, we obtain:
 $$\udim{R^{\phi}}{R}
\leq \udim{Z}{R[x;\phi]}=\udim{ZZ^{-1}}{Q(R[x;\phi])} < \infty .$$
This gives the statement $(a)$.

 (b). Suppose that  every regular element of $Z(R)^{\phi}$ is regular in $Z(R)$.
 That is $B=Z(R)(Z(R)^{\phi})^{-1}$ and
 $Z(R)^{\phi}(Z(R)^{\phi})^{-1}$ means localizations with respect
 the same Ore set of all regular elements of $Z(R)^\phi$.

 Notice   that,  in order to prove
the statement $(b)$, it is enough to show
 that $Z(R)Z(R)^{-1}=B$, as  $Q(R)=RZ(R)^{-1}$.
Since the ACC on annihilators is a hereditary condition on
subrings and $Z(R)$ is a reduced ring (i.e. $Z(R)$ does not
contain nontrivial nilpotent elements), $Z(R)^{\phi}$ is a
commutative reduced ring with the ACC on annihilators. Therefore,
its classical quotient ring
$A=Z(R)^{\phi}(Z(R)^{\phi})^{-1}\subseteq B$ is a finite product
of fields, say $A=\bigoplus_{i=1}^nK_i$, where $K_i=e_iA$ for
suitable primitive orthogonal idempotents, $1\leq i \leq n$.

 Recall that $\phi$ is either an automorphism or a derivation of
 $R$. Hence  $Z(R)$ is
stable by $\phi$ and we can consider the Ore extension
$Z(R)[x;\phi]$ and, since $R[x;\phi]$ is a \mbox{\rm PI} ring,
$Z(R)[x;\phi]$ is also a \mbox{\rm PI} ring . Now,  we can apply
the statement $(2)(a)$  to $Z(R)$ and get
$\udim{Z(R)^{\phi}}{Z(R)}<\infty$. Consequently, by Lemma
\ref{udim and localizations}, we get $\udim{A}{B}<\infty$. This
implies that, for any $1\leq i\leq n$, $e_iB$ is a finite
dimensional algebra over the field $K_i=e_iA$. Therefore every
regular element in $e_iB$ is invertible in $e_iB$. This, in turn,
implies that every regular element of $B=\bigoplus_{i=1}^ne_iB$ is
invertible in $B=Z(R)(Z(R)^{\phi})^{-1}$. This shows that $B$ is
equal to $Z(R)Z(R)^{-1}$ and completes the proof of the lemma.
\end{proof}
\begin{remark} Let us observe that if $R$ is a prime ring, then
the assumption of the statement (2b) from the above proposition is
always satisfied.

\end{remark}
\section{\sc \normalsize Prime Coefficient Ring}

In this section $R$ will stand for a prime \mbox{\rm PI} ring. We
first continue to gather some information on the behaviour of
$\si$ on the center.

\begin{prop} \label{inner auto}
 Let $R$ be a prime \mbox{\rm PI} ring and $\si$ an endomorphism of $R$ such
 that $\si^n$ is an automorphism of the center $Z(R)$ of $R$. Then:
\begin{enumerate}
  \item $\si$ extends uniquely to an automorphism of the
  localization $RZ(R)^{-1}$.
  \item Suppose additionally that  $\si^n|_{Z(R)}=\mbox{\rm id}_{Z(R)}$.
  Then there is $0\ne u\in Z(R)$ such that
   $\si(u)=u$,  $\si$ is an automorphism of
   $RS^{-1}$ and $\si^n$ is an inner automorphism of
   $RS^{-1}$, where $S=\{u^k\mid k\geq 0\}$.
  \end{enumerate}
\end{prop}
\begin{proof} (1) By Lemma \ref{generalities on the center},
$\si$ is injective  and $\si(Z(R))=Z(R)$. Thus $\si$ has a unique
extension to the localization $RZ(R)^{-1}$ which, by Posner's
Theorem, is a simple,  finite dimensional algebra over the center
$Z(R)Z(R)^{-1}$.
 Now the statement (1) is a direct consequence of  Lemma
 \ref{generalities on the center}$(3)$.

 (2) By (1), $\si$ is an automorphism of $RZ(R)^{-1}$  and the theorem of Skolem-Noether
 implies that   $\si^n$ is an inner automorphism of $RZ(R)^{-1}$. Therefore,
  one can choose a
 regular element  $a\in R$ with the inverse $bv^{-1}\in RZ(R)^{-1}$ such that
 $\si^n(r)=arbv^{-1}$ for all $r\in R$.
  Let $u=v\si(v)\ldots \si^{n-1}(v)$. Then
 $\si(u)=u$, the element $bu^{-1}$ has the inverse $a\si(v)\ldots\si^{n-1}(v)$
 and also  determines $\si^n$. This
 yields the thesis.
\end{proof}

In the sequel we  will say that an automorphism $\si$ of a ring
$R$ is of finite inner order if $\si^n$ is an inner automorphism
of $R$, for some $n\geq 1$.

 At this point we make a small
digression not directly related to the main theme of the paper. It
is known (Cf. \cite{Jo}) that for any ring $R$ with a fixed
injective endomorphism $\si$ there exists a universal over-ring
$A$ of $R$, called a Jordan extension of $R$, such that $\si$
extends to an automorphism of $A$ and $A= \bigcup_{i=0}^\infty
\sigma^{-i}(R)$. In this case we will write $R\subseteq_\si A$.

 It is easy to check that
if $\si$ becomes an inner automorphism of $A$, then $R=A$. Also,
 if $R$ is a prime PI ring, then  $A$ is prime PI as well.

 Recall that
an automorphism of a prime \mbox{\rm PI} ring $R$ is $X$-inner if
and only if it becomes inner when extended to the classical
quotient ring $Q(R)=RZ(R)^{-1}$.

Suppose that $R$ is a prime \mbox{\rm PI} ring and $\si$ is an
endomorphism of $R$ such that $\si^n|_{Z(R)}=\mbox{id}|_{Z(R)}$.
Then, by Proposition \ref{inner auto}, $R\subseteq_{\si}A\subseteq
RS^{-1} $, where $S$ consists of powers of a single central
element. Moreover $\si$ is an $X$-inner automorphism of $A$. The
following example shows, that all inclusions $R\subseteq_{\si}
A\subseteq RS^{-1} $ can be strict.
  This example will be used again later in the
paper.

\begin{example}
 \label{example}
 Let $R=\left[\begin{array}{ll}
   \mathbb{Z}+x\mathbb{Q}[x] & \mathbb{Z}+x\mathbb{Q}[x] \\
   x\mathbb{Q}[x] &  \mathbb{Z}+x\mathbb{Q}[x] \
\end{array}\right]$
be a subring of $M_2(\mathbb{Q}[x])$ and $\si$ denote the inner
automorphism of $M_2(\mathbb{Q}[x])$ adjoint to the element
  $u=\left[ \begin{tabular}{cc}
 1 & 0 \\
  0 & 2 \\
\end{tabular}\right]$.
 Then
$\si(\left[\begin{tabular}{cc}

  $a$ & $b$ \\
  $c$ & $d$ \\
\end{tabular}\right])= \left[\begin{tabular}{cc}

  $a$ & $2b$ \\
  $\frac{1}{2}c$ & $d$ \\
\end{tabular}\right]$.
This means that the restriction of $\si$ to $R$ is an injective
endomorphism of $R$ which is not onto. One can check that
$R\subseteq_{\si} A=\left[\begin{tabular}{ll}
  $ \mathbb{Z}+x\mathbb{Q}[x]$ & $\mathbb{Z}[\frac{1}{2}]+x\mathbb{Q}[x] $\\
  $ x\mathbb{Q}[x]$ &$  \mathbb{Z}+x\mathbb{Q}[x]$ \
\end{tabular}\right]$  and $\si$ becomes an inner automorphism on
the localization $RS^{-1}$, where $S$ denotes the multiplicatively
closed  set generated by $2$.
\end{example}

\begin{definition}
Let $R$ be a ring, $\si$ an endomorphism of $R$ and $\de$ a
$\si$-derivation of $R$. We say that the center of the Ore
extension $\T{R}$ is nontrivial if it contains a nonconstant
 polynomial.
\end{definition}
\begin{lemma}\label{si identity on center} Let $\si$ be an injective endomorphism
of a prime ring $R$.
  Then:
  \begin{enumerate}
  \item If  $\Tn$ is a  \mbox{\rm PI} ring, then the center $Z(\T{R})$ is nontrivial.
  \item  Let
  $f=ax^n+...\in Z(R[x;\si,\de])$
  be a polynomial of degree $n\geq 1$.
  Then $a$ is a regular element of $R$,
  $\si^n|_{Z(R)}=\mbox{\rm id}_{Z(R)}$ and
   $\sigma\vert_{Z(R)}$ is an automorphism of $Z(R)$.
  \end{enumerate}

\end{lemma}
\begin{proof}
  (1).  Suppose that $\Tn$ is a  \mbox{\rm PI} ring. Then Lemma \ref{known facts}
 and Proposition  \ref{classics} imply that $\T{R}$ is a prime \mbox{\rm PI} ring. Thus
 every essential one-sided
 ideal contains a nonzero central element. Since $\si$ is injective,
 the element $x$ is regular in $\T{R}$.
 Therefore $\T{R} x$ contains a nonzero central element $f=ax^n+a_{n-1}x^{n-1}+\ldots
 +a_1x$, where $a\ne 0$ and $n\geq 1$  and (1) follows.

 (2)  Let
  $f=ax^n+...\in Z(R[x;\si,\de])$ be
  such that $a\not= 0$ and $n\ge1$.
  Making use of  $xf=fx$ and $rf=fr$, for any $r\in R$,
   we obtain $\si(a)=a$ and $ra=a\si^n(r)$, for any $r\in R$.

  We claim that $a$ is a regular element in $R$. Indeed, if $b\in
  R$ is such that $ba=0$, then $bRa=ba\si^n(R)=0$. Hence $b=0$, as $R$ is a prime
  ring. Thus $a$ is left regular. If $ab=0$, then
  $0=\si^n(a)\si^n(b)=a\si^n(b)=ba$. Since $a$ is left regular,
  $b=0$ follows.

  Now, for any $z\in Z(R)$ we have $az=za=a\si^n(z)$. Thus
  $a(z-\si^n(z))=0$ and $\si^n(z)=z$, for any $z\in Z(R)$,
 follows. The last assertion of (2) is then a consequence
 of Lemma \ref{generalities on the center}(1).
\end{proof}
\begin{prop}\label{prime PI endoporphism}
 Let $R$ be a prime \mbox{\rm PI} ring and $\si$ an injective endomorphism of $R$. The
 following conditions are equivalent:
\begin{enumerate}
  \item $R[x;\si]$ is a \mbox{\rm PI} ring.
  \item $\si|_{Z(R)}$ is an automorphism of $Z(R)$ of finite order.
  \item There exists  $0\ne u\in Z(R)$ such that $\si(u)=u$ and $\si$
   is an  automorphism of finite inner order of the localization $RS^{-1}$ , where
  $S$ denotes the set of all powers of $u$.
\end{enumerate}
\end{prop}
\begin{proof}
 The implications $(1)\Rightarrow (2)$ and
 $(2)\Rightarrow (3)$ are  given by  Lemma \ref{si identity on
 center} and Proposition \ref{inner auto}, respectively.

 $(3)\Rightarrow (1)$. Suppose that (3) holds and let $\si^n$, $n\geq 1$, be an inner
 automorphism of $RS^{-1}$. By the choice of $S$, the set $S$ is
 central in $R[x;\si]$ and $(R[x;\si])S^{-1}=RS^{-1}[x;\si]$.
 Since $\si^n$ is an inner automorphism of $RS^{-1}$, the subring
 $RS^{-1}[x^n]\subseteq RS^{-1}[x;\si]$ is isomorphic to the usual
 polynomial ring $RS^{-1}[y]$, so it satisfies a polynomial
 identity, as $RS^{-1}$ is a \mbox{\rm PI} ring.  Now, the fact that
 $RS^{-1}[x;\si]$ is a finitely generated free module over the \mbox{\rm PI}
 subring $RS^{-1}[x^n]$ implies that $R[x;\si]$ satisfies a
 polynomial identity.
\end{proof}

In case $\si$ is an automorphism of $R$, the equivalence $(1)
\Leftrightarrow (2)$ in the above proposition was also obtained by
Pascaud and Valette (Cf. \cite{PV}) using another approach.

Before stating the next results we need to recall some definitions
(Cf.\cite{LM2}):
\begin{definition}
Let $R$ be a ring, $\si$ an endomorphism  and $\de$ a
$\si$-derivation of $R$, respectively. We say that:
\begin{enumerate}
\item   $\de$ is quasi
algebraic if there exists  $n\ge 1$ and elements $b, a_1,
\dots,a_{n}\in R$, with $a_n\not= 0$, such that
$\sum_{i=1}^na_i\de^i=\de_{b,\si^n}$ where $\de_{b,\si^n}$ denotes
the inner $\si^n$-derivation adjoint  to the element $b$, that is
 $\de_{b,\si^n}(r)=br-\si^n(r)b$, for any $r\in R$.
\item A polynomial $p \in \T{R}$ is right semi-invariant if for any
element $a\in R$ there exists $b\in R$ such that $pa=bp$.
\end{enumerate}
\end{definition}

\begin{theorem}\label{main prime}
 Suppose $\si$ is an injective endomorphism of a prime \mbox{\rm PI} ring
 $R$.  Let $Q(R)=RZ(R)^{-1}$ denote the classical quotient ring
 of $R$.  The following conditions are equivalent:
\begin{enumerate}
  \item $\T{R}$ is a \mbox{\rm PI} ring.

\item There exists a nonconstant central polynomial in $\T{R}$
  with a regular leading coefficient.
  \item The center of $\T{R}$ is nontrivial.
  \item The center of $\T{Q(R)}$ is nontrivial.
  \item There exists a nonconstant central polynomial in $\T{Q(R)}$
  with invertible leading coefficient.
  \item $Q(R)[x;\si ,\de]$ is a \mbox{\rm PI} ring.
  \item  $\si$ is an automorphism
    of $Q(R)$ of finite  inner order and
  $\de$ is a quasi algebraic $\si$-derivation of $Q(R)$.
  \item  $\si$ is an automorphism
    of $Q(R)$ of finite  inner order and
$Q(R)[x;\si ,\de]$ contains a monic nonconstant semi-invariant
polynomial.
\end{enumerate}
\end{theorem}
\begin{proof}
 The implication $(1)\Rightarrow (2)$ is given by Lemma \ref{si identity on
 center}. The implication $(2)\Rightarrow (3)$ is clear.

 $(3)\Rightarrow (4)$. Suppose that the center of $\Tn$ is
 nontrivial. Then, by Lemma \ref{si identity on
 center}(2), $\si(Z(R))=Z(R)$. This means that we can
 extend uniquely $\si$ and $\de$ to an endomorphism and  a
 $\si$-derivation of $Q(R)=RZ(R)^{-1}$, respectively, and consider
  the over-ring $Q(R)[x;\si,\de]$ of $\Tn$.
 It is standard to check that $Z(\Tn)\subseteq Z(\T{Q(R)})$.
 This shows that the center of $\T{Q(R)}$ is nontrivial.

 The implication $(4)\Rightarrow (5)$ is given by Lemma \ref{si identity on
 center}(2) applied to the ring $Q(R)$ and the fact that every regular element
 of $Q(R)$ is invertible.

$(5)\Rightarrow (6)$.
  Let $f\in \T{Q(R)}$ be a central
  polynomial of degree $n\geq 1$ with an invertible leading
  coefficient. Then the subring $Q(R)[f]\subseteq \T{Q(R)}$ is
  isomorphic to the usual polynomial ring $Q(R)[y]$ in one
  indeterminate $y$, so the ring $Q(R)[f]$ satisfies a polynomial identity.
  Notice also that, due to Lemma \ref{si identity on
 center}(2) and Proposition \ref{inner auto}, $\si$ is an
 automorphism of $Q(R)$.
  Now, the fact that $\T{Q(R)}$ is a free module over $Q(R)[f]$
  with basis $1, x, \ldots, x^{n-1}$ implies that $\T{Q(R)}$ is a
  \mbox{\rm PI} ring.

The implication $(6)\Rightarrow (1)$ is obvious.

The above shows that conditions $(1)\div (6)$ are equivalent and
that the extension of $\si$ to $Q(R)$ is an automorphism of
$Q(R)$.

Now the equivalence of statements $(4),(7)\,\mbox{\rm and}(8)$ is
given by Theorem 3.6 in \cite{LM2}.
\end{proof}
\begin{remark}\label{rem 1}
 (1). Notice that due to Proposition \ref{inner auto}(2), the assumption in (7)
 and (8) of the above theorem that  $\si^n$ is an inner automorphism
 of $Q(R)$, for some $n\geq 1$, can be replaced by a condition
 that $\si|_{Z(R)}$ is an automorphism of $Z(R)$
 of finite order.

 (2). One can also
 replace the ring $Q(R)$ in the above theorem by a  localization
$RS^{-1}$ where $S$ denotes a multiplicatively closed set
consisting of all powers of a suitably chosen central
$\si$-invariant element.
\end{remark}

Theorem \ref{main prime} says that if $\Tn$ is a PI ring, then the
extension of  $\si$ to $Q(R)$ has to be an automorphism of $Q(R)$.
It is not a surprise as, by Corollary D\cite{BH}, every
endomorphism of a semiprime Noetherian PI ring $R$ which is
identity on $Z(R)$ is always an automorphism of $R$. Nevertheless,
when $R$ is a prime PI ring (so it satisfies the ACC on
annihilators), then the injective endomorphism $\si$ of $R$ does
not have to be onto when $\Tn$ is a PI ring.

\begin{example}
 Let $R$ and $\si$ be as in Example \ref{example}. Then $R$ is a
 prime ring, the injective endomorphism  $\si$ is not onto
 and $R[y;\si]$ satisfies
 a polynomial identity, since $R[y;\si]\subseteq M_2(Q[x])[y;\si]\simeq M_2(Q[x])[z]$.
\end{example}

\begin{theorem}\label{prime case}
Let $R$ be a prime ring, $\si$ an injective endomorphism and $\de$
a $\si$-derivation of $R$. Then $\T{R}$ is a {\rm PI} ring if and
only if $R$ is {\rm PI}, $\si|_{Z(R)}$ is an automorphism of
$Z(R)$ of finite order and one of the following conditions holds:
\begin{enumerate}
\item If $\mbox{\rm char}R=0$,
   $$Q(R)[x;\si,\de]\simeq \left\{
   \begin{array}{lll}
    Q(R)[x;\si] & if &  \si|_{Z(R)}\ne \mbox{\rm id}_{Z(R)}\\
      Q(R)[x] & else &\\
   \end{array} \right. $$
\item If $\mbox{\rm char}R=p\ne 0$,
   $$Q(R)[x;\si,\de]\simeq \left\{
   \begin{array}{lll}
    Q(R)[x;\si] & if &  \si|_{Z(R)}\ne \mbox{\rm id}_{Z(R)}\\
      Q(R)[x;d] & else &\\
   \end{array} \right. $$
   where $d$ is a suitable derivation of $Q(R)=RZ(R)^{-1}$ such that:
\begin{enumerate}
  \item $d(R)\subseteq R$
  \item there exist elements
   $q_l=1,\ldots, q_1\in Q(R)$ such that $\sum_{i=0}^{l}q_id^{p^i}$
   is an inner derivation of $Q(R)$
 \end{enumerate}
\end{enumerate}
\end{theorem}
\begin{proof} Suppose $R[x;\si,\de]$ is a \mbox{\rm PI} ring. Then $R$ is \mbox{\rm PI}
and Theorem \ref{main prime} and Remark \ref{rem 1} show that
$\T{Q(R)}$ is a \mbox{\rm PI} ring, $\si$ is an automorphism of
$Q(R)$ and $\si|_{Z(R)}$ is an automorphism of the center $Z(R)$
of finite order.

 Suppose that $\si|_{Z(R)}\ne \mbox{\rm id}|_{Z(R)}$ and let $c\in Z(R)$ be
 such that $\si(c)\ne c$. Then it is well-known that $\de$ is an
 inner $\si$-derivation of $Q(R)$ adjoint to the element
 $a=(c-\si(c))^{-1}\de(c)$.
 Then $Q(R)[x;\si,\de]=Q(R)[x-a;\si]\simeq Q(R)[x;\si]$.

  Suppose $\si|_{Z(R)}= \mbox{\rm id}_{Z(R)}$.
 Then, by Proposition
\ref{inner auto}(2), $\si$ is an inner automorphism of $Q(R)$, say
induced by an invertible element $c\in Q(R)$, i.e.
$\si(r)=crc^{-1}$, for any $r\in Q(R)$. Since $Q(R)$ is a central
localization of $R$ we can write $c^{-1}=uz^{-1}$ for some $u\in
R$ and $z\in Z(R)$ and we have $\si(r)=u^{-1}ru$ for all $r\in Q$.
Then $u\de=d$ is a derivation of $Q(R)$ such $d(R)\subseteq R$ and
  $Q(R)[x;\si,\de]=Q(R)[ux;\mbox{\rm id},u\de]\simeq Q(R)[x;d]$.

  Applying Theorem \ref{main prime} to the \mbox{\rm PI}  ring
  $Q(R)[x;d]$ we know that this ring contains a monic nonconstant
  semi-invariant polynomial. Now the thesis is a consequence of
  Proposition 2.8 and Lemma 2.3 from \cite{LM2} and the fact that
  $Q(R)[x;d]$ is isomorphic to $Q(R)[x]$, provided $d$ is an inner derivation.

  Conversely, if $R$ is a prime \mbox{\rm PI} ring and $\si^n$ is the
  identity on the center, then Theorem \ref{main prime}(8), Remark \ref{rem 1}
  and the hypothesis
  made on $Q(R)[x;\si,\de]$ shows that $R[x,\si,\de]$ is \mbox{\rm PI}.
\end{proof}

\begin{remark}
 In case $\si$ is an automorphism of $R$, the statement (1) from
the above Theorem is exactly the result of Jondrup \cite{Jon}, see
also the book by Goodearl and Brown \cite{BGO}.
\end{remark}

We have seen that the center $Z(R)$ of $R$ plays a crucial role in
determining if an Ore extension satisfies a polynomial identity.
This theme will be pursued further in the next  result and in
Section 3.

For any subring $A$ of $R$, $A^{\si,\de}$ will denote the
 the subalgebra of $(\si,\de)$-constants, i.e.
 $A^{\si,\de}=\{a\in A\mid \si(a)=a\mbox{ and } \de(a)=0\}$. Notice that we do not
 require that $A$ is  $\si$ or $\de$ stable.

With the above notations we have :

\begin{theorem}
\label{prime caracterization with udim} Suppose that $\si$ is an
injective endomorphism  of a prime {\rm PI} ring $R$ with the
center $Z$. Let $K$ denote the field of fractions of
$Z^{\si,\de}$. The following conditions are equivalent:
\begin{enumerate}
  \item $\T{R}$ is a  {\rm PI} ring.
  \item $\udim{Z^{\si,\de}}{Z}<\infty$
  \item $\dim_{K}Q(R)<\infty$ and $Q(R)=R(Z^{\si,\de})^{-1}$.
  \item $\dim_{K}Q(R)<\infty$.
  \end{enumerate}
\end{theorem}
\begin{proof}
 $(1)\Rightarrow (2)$ Suppose that $\Tn$ is a \mbox{\rm PI} ring.
 Then, by  Theorem \ref{prime case},
 $\T{Q(R)}$ is  $Q(R)$-isomorphic to $Q(R)[x;\phi]$, where either
  $\phi=\si$ and $\phi(Z)=Z$  or $\phi$ is a derivation
  of $Q(R)$ such that $\phi(R)\subseteq R$. In particular, in any case we have
  $\phi(Z)\subseteq Z$ and  we
can consider $Z[x;\phi]$ as a subring of the \mbox{\rm PI} ring
$Q(R)[x,\phi]$. Now, Lemma \ref{the untwisted cases}(2) shows that
$\udim{Z^{\phi}}{Z}$ is finite.

   Notice that:
 $$Z(Q(R))^{\si,\de}=Q(R)\cap Z(\T{Q(R)})=Q(R)\cap
 Z(Q(R)[x;\phi])=Z(Q(R))^{\phi}.$$
 Therefore, $Z^{\si,\de}=Z\cap Z(Q(R))^{\si,\de}=Z\cap
 Q(R)^{\phi}=Z^{\phi}$ and  $\udim{Z^{\si,\de}}{Z}=\udim{Z^{\phi}}{Z}<\infty$
 follows.

 $(2)\Rightarrow (3)$. Suppose that
 $\udim{Z^{\si,\de}}{Z}<\infty$. Then, making use of
 Lemma \ref{udim and localizations}, we obtain:
 $\dim_{K}Z(Z^{\si,\de})^{-1}= \udim{Z^{\si,\de}}Z< \infty$.  This
 implies that the commutative domain
 $Z(Z^{\si,\de})^{-1}$ is a field.  Therefore
 $Z(Z^{\si,\de})^{-1}=ZZ^{-1}$ and $Q(R)=RZ^{-1}=R(Z^{\si,\de})^{-1}$.
  Since $Q(R)$ is finite dimensional over its center $ZZ^{-1}=Z(Z^{\si,\de})^{-1}$
  which is a finite
 dimensional field extension of $K$, we have
 $\dim_KQ(R)<\infty$.

 The implication $(3)\Rightarrow (4)$  is clear.

 $(4)\Rightarrow (1)$.  Under the hypothesis $(4)$,
 $\T{Q(R)}$ is a finitely generated module over the commutative
 polynomial ring $K[x]$. This yields   that $\T{Q(R)}$ is a \mbox{\rm PI} ring and so is $\Tn$.
\end{proof}

\section{ \sc \normalsize  Semiprime Coefficient Ring}

In this section we  will investigate Ore extensions $\Tn$ over a
semiprime coefficient ring $R$ satisfying the ACC on annihilators.
 We will frequently use the
following lemma, which is an obvious application of results of
Cauchon and Robson from \cite{CR} (Cf. Lemma 1.1 to Lemma 1.4).

\begin{lemma}
\label{Cauchon and Robson} Let $R=\bigoplus_{i=1}^s B_i$ be a
decomposition of a semisimple ring $R$ into its simple components
and let $\si,\de$ be an injective endomorphism and a
$\si$-derivation of $R$, respectively. Then :
\begin{enumerate}
\item There exists a permutation $\rho$ of the index set $\{1,\dots,s\}$
such that $\si (B_i)\subseteq B_{\rho(i)}$ and $\de (B_i)\subseteq
B_i + B_{\rho (i)}$.
\item If $\{1,\dots,s\}=\cup_{j=1}^k {\cal O}_j$ is the
decomposition of the index set into orbits under the action of the
permutation $\rho$ and $A_j:=\bigoplus _{i\in {\cal O}_j}B_i$,
then $R[x;\si,\de]=\bigoplus_{j=1}^k
A_j[x_j;\si\vert_{A_j},\de\vert_{A_j}]$.
\item Let $j\in \{1,\dots,k\}$ be such that $\vert\,{\cal O}_j\, \vert
>1$, then $\de\vert_{A_j}$ is an inner $\si|_{A_j}$-derivation of
$A_j$. In particular, $A_j[x_j;\si\vert_{A_j},\de\vert_{A_j}]$ is
$A_j$-isomorphic to $A_j[x_j;\si|_{A_j}]$.
\item There exists an $m\ge 1$ such that $\si^m(B_i)\subseteq
B_i$, for all $i\in \{1,\dots,s\}$.
\end{enumerate}
\end{lemma}

Let us first consider the case of an Ore extension of endomorphism
type.

\begin{prop}\label{semiprime endomorhism}
 Let $\si$ be an
 injective endomorphism of a semiprime \mbox{\rm PI} ring $R$ satisfying the \mbox{\rm ACC}
  on annihilators. The
 following conditions are equivalent:
\begin{enumerate}
  \item $R[x;\si]$ is a \mbox{\rm PI} ring.
  \item The restriction $\si\vert_{Z(R)}$ is an automorphism of $Z(R)$ of finite order.
  \item There exists $n\geq 1$ such that $\si^n$ is identity on
  the center $Z(R)$ of $R$.
  \item There exists   a regular element $ u\in Z(R)$ such
  that $\si(u)=u$ and $\si$
  is an  automorphism of the localization $RS^{-1}$ of finite inner order, where
  $S$ denotes the set of all powers of $u$.
  \item  $\si$ is an
  automorphism of $Q(R)$ of finite inner order.
\end{enumerate}
\end{prop}
\begin{proof} By Lemma \ref{known facts}, $\si$ can be extended to
an injective endomorphism of  $Q(R)=RZ(R)^{-1}$. Let
$Q(R)=\bigoplus_{i=1}^sB_i$ be a decomposition of $Q(R)$ into its
simple components.

 $(1)\Rightarrow (5)$. Suppose  $R[x;\si]$ is a PI ring. Thus, by Proposition
 \ref{reduction to semisimple},
  $Q(R)[x;\si]$ also  satisfies
 a polynomial identity.
 By Lemma \ref{Cauchon and Robson}, there exists $m\geq 1$, such
that all components $B_i$ are $\si^m$-stable.
 Therefore,
 $Q(R)[x;\si^m]\simeq Q(R)[x^m;\si^m]\subseteq Q(R)[x;\si]$
 is also a PI ring and
$$Q(R)[x;\si^m]=(\bigoplus_{i=1}^sB_i)[x;\si^m]\simeq\bigoplus_{i=1}^sB_i[x_i;\si^m].$$
This shows that, for any $1\leq i\leq s$, $B_i[x_i;\si^m]$
satisfies a polynomial identity and Theorem \ref{main prime}(8)
applied to each simple component $B_i$ yields that there exists
$k\geq 1$ such that  $\si^{n}$, where $n=mk$, is an inner
automorphism of $Q(R)$, i.e. $(5)$ holds.

 Since $Q(R)=RZ(R)^{-1}$, the implication $(5)\Rightarrow (4)$
 can be proved using the same argument
 as in the proof of  Proposition \ref{inner auto}(2).

The implication $(4)\Rightarrow (3)$ is clear and $(3)\Rightarrow
(2)$  is a direct consequence of Lemma \ref{generalities on the
center}(1).

$(2)\Rightarrow (1)$. Suppose that $(2)$ holds and let $n$ denote
the order of $\si|_{Z(R)}$. Then $\si^n(B_i)\subseteq B_i$,  for
$1\leq i\leq s$. Then, by Lemma \ref{generalities on the center}
and the theorem of Skolem-Noether, $\si^n|_{B_i}$ is an inner
automorphism of $B_i$, for any $i$. Hence $\si^n$ is an inner
automorphism of $Q(R)$ and the subring $Q(R)[x^n]\subseteq
Q(R)[x;\si]$ is isomorphic to a polynomial ring $Q(R)[y]$, so it
satisfies a polynomial identity. This implies that $Q(R)[x;\si]$
is a PI ring, as $Q(R)[x;\si]$ is a finitely generated free module
over its subring $Q(R)[x^n]$ and (1) follows.
\end{proof}

As an immediate application of the above proposition, Corollary
\ref{graded PI} and Proposition
 \ref{reduction to semisimple} we easily get the following:

\begin{corollary}\label{derivation}
Let $R$ be a   semiprime ring  with \mbox{\rm ACC} on annihilators
and $\si$ an injective endomorphism of $R$. If $\T{R}$ is a
\mbox{\rm PI} ring then $Q(R)[x;\si,\de]$ is a \mbox{\rm PI} ring
and  $\si$ is an automorphism  of the semisimple ring $Q(R)$ of
finite inner order.
\end{corollary}

 Thus while investigating the \mbox{\rm PI} property of $\T{R}$, the crucial
  case is when the ring $R$ is semisimple. Moreover, by Lemma \ref{Cauchon and
  Robson}(2), one may restrict attention to the case when $\si$
  acts transitively on the set of all simple components $B_i$ of
  $R=\bigoplus_{i=1}^sB_i$.
\begin{prop}\label{char=0}
 Let $R= \bigoplus_{i=1}^sB_i$ be a decomposition of  a semisimple \mbox{\rm PI} ring
  into simple components and $\si$  an injective endomorphism
 of $R$. Suppose that  $\si$ acts transitively on the set of simple components.
   Then the following conditions are equivalent:
\begin{enumerate}
  \item $\T{R}$ is a \mbox{\rm PI} ring.
  \item $\si$ is an automorphism of $R$ of finite inner order and
  one of the following conditions holds:
\begin{enumerate}
  \item $\Tn$ is isomorphic either to $R[x;\si]$ or to $R[x]$.
  \item $R$ is a simple ring  of a nonzero characteristic $p$ and
  $R[x;\si,\de]\simeq R[x;d] $
   where $d$ is a  derivation of $R$ such that there
   exist elements  $q_l=1,\ldots, q_1\in R$
   such that $\sum_{i=0}^{l}q_id^{p^i}$ is an inner derivation of
   $R$.
\end{enumerate}
\end{enumerate}
\end{prop}
\begin{proof} If $R$ is simple, i.e. $R=B_1$, then the proposition is a
direct consequence of Theorem \ref{prime case}.

If $R$ is not simple then, by Lemma \ref{Cauchon and Robson}(3),
 $\Tn\simeq R[x;\si]$ and
the proposition is a consequence of Proposition \ref{semiprime
endomorhism}.
\end{proof}
\begin{corollary}\label{C semisimple}
 Let $R$ be a semisimple \mbox{\rm PI} ring with an injective
 endomorphism $\si$. The following conditions are equivalent:
\begin{enumerate}
  \item $\Tn$ is a \mbox{\rm PI} ring.
  \item The center of  $\Tn$ contains a nonconstant  polynomial with
  invertible leading coefficient.
  \item $\udim{Z(R)^{\si,\de}}{Z(R)}$ is finite.
 \end{enumerate}
\end{corollary}
\begin{proof}
$(2)\Rightarrow (1)$. Let $f\in \Tn$ be a nonconstant polynomial
from the center of $\Tn$ with invertible leading coefficient. Then
the subring $R[f]\subseteq \Tn$ satisfies a polynomial identity
and $\Tn$ is a finitely generated left module over $R[f]$. This
implies that $\Tn$ is a PI ring.

$(1)\Rightarrow (2)$. Let $R= \bigoplus_{i=1}^s B_i =
\bigoplus_{j=1}^kA_j$, where $A_j=\bigoplus _{i\in {\cal
O}_j}B_i$, be
 a decomposition of $R$ described in Lemma \ref{Cauchon and Robson}.
 Then, by the same lemma,  we have  $\T{R}=\bigoplus_{j=1}^k A_j[x_j;\si_j,\de_j]$,
 where $\si_j=\si|_{A_j}$ and $\de_j=\de|_{A_j}$. Hence, by
 Proposition \ref{char=0}, the PI ring $T_j=A_j[x_j;\si_j,\de_j]$ is
 $A_j$-isomorphic to one of the following rings: $A_j[x]$,
 $A_j[x;\si_j]$, where $\si_j$ is an automorphism
 of $A_j$ of finite inner order or $A_j[x;d_j]$, where $d_j$ is a derivation of a
 simple ring $A_j$.

Suppose $T_j\simeq A_j[x;\si_j]$ and let $n\geq 1$ and $u\in A_j$
be an invertible element such that, for any $a\in T_j$,
$\si^n(a)=u^{-1}au$. It is known that, eventually replacing $n$ by
$n^2$ and $u$ by $u\si(u)\ldots \si^{n-1}(u)$, we may additionally
assume that $\si(u)=u$. Then $f_j=ux^n$ is a  polynomial from the
center of $T_j$.

Suppose that $T_j\simeq A_j[x;d_j]$, where $A_j$ is a simple ring.
Then, by Theorem \ref{main prime}, $T_j$ contains a nonconstant
central polynomial $f_j$ with an invertible leading coefficient.

The above shows that, for any $1\leq j\leq k$,
$T_j=A_j[x_j;\si_j,\de_j] $ contains a  nonconstant central
polynomial $f_j$ with  invertible leading coefficient. Since a
power of a central element is again central, we may choose the
polynomials $f_j$'s in such a way that $\deg f_i=\deg f_j$, for
all $1\leq i,j\leq k$. Then the polynomial $f=\sum_{j=1}^kf_j$
belongs to the center of $\Tn$, is nonconstant and the leading
coefficient of $f$ is invertible.

$(1)\Leftrightarrow (3)$. We will continue to use the notation as
in the proof of $(1)\Rightarrow (2)$. Notice that
$Z(R)=\bigoplus_{j=1}^k Z(A_j)$ and
$Z(R)^{\si,\de}=\bigoplus_{j=1}^k Z(A_j)^{\si_j,\de_j}$. Hence
$\udim{Z(R)^{\si,\de}}{Z(R)}=\sum_{j=1}^k\udim{Z(A_j)^{\si_j,\de_j}}{Z(A_j)}$.
This means that, without loosing generality, we may assume that
$R=A_1$, i.e. $\si$ acts transitively on the simple components of
$R$.

If $R$ is simple, then the equivalence $(1)\Leftrightarrow (3)$ is
given by Theorem \ref{prime caracterization with udim}.

Suppose $R$ is not simple. Then, by Lemma \ref{Cauchon and
Robson}(3), $\Tn$ is $R$-isomorphic to $R[x;\si]$. Since
$Z(R)^{\si,\de}=Z(\Tn)\cap R$ and $Z(R)^\si=Z(R[x;\si])\cap R$, we
may replace  $\Tn$ by $R[x;\si]$.

 If $R[x;\si]$ satisfies a
polynomial identity then Lemma \ref{the untwisted cases}(2a),
applied to $Z(R)[x;\si]$, shows that $\udim{Z(R)^\si}{Z}< \infty$.

Suppose now, that $\udim{Z(R)^\si}{Z}$ is finite. Recall that
$R=\bigoplus_{i=1}^sB_i$ and, by Lemma \ref{Cauchon and
Robson}(1), there is $n\geq 1$, such that $\si^n(B_i)\subseteq
B_i$, for all $i$ and
$R[x;\si^n]=\bigoplus_{i=1}^sB_i[x_i;\tau_i]$, where
$\tau_i=\si^n|_{B_i}$. Since $Z(R)^\si\subseteq Z(R)^{\si^n}$,
$\udim{Z(R)^{\si^n}}{Z}<\infty$ and, consequently,
$\udim{Z(B_i)^{\tau_i}}{Z(B_i)}<\infty$, for any $1\leq i\leq s$.
Therefore,  Theorem \ref{prime caracterization with udim} applied
to Ore extensions $B_i[x_i;\tau_i]$ shows that
$R[x;\si^n]=\bigoplus_{i=1}^sB_i[x_i;\tau_i]$ is a PI ring. Since
$R[x;\si]$ is a finitely generated module over its subring
$R[x^n]$ which is itself isomorphic to the PI ring $R[x;\si^n]$,
we conclude that $R[x;\si]$ is a PI ring.
\end{proof}

The following lemma  is of crucial importance for the forthcoming
theorem.
\begin{lemma}  \label{le} Suppose that  the ring $\Tn$ satisfies a polynomial
identity, where $R$ is a semiprime  ring with the \mbox{\rm ACC}
on annihilators and $\si$ is an injective endomorphism of $R$. Let
$Z$ denote the center of $R$ and $Q=Q(R)$. Then:
\begin{enumerate}
  \item $Q=R(Z^{\si,\de})^{-1}$
  \item If an  element $a\in Z^{\si,\de}$ is regular in $Z^{\si,\de}$, then
  $a$ is regular in $R$.
  \item $Z(Q)=Z(Z^{\si,\de})^{-1}$ and $Z(Q)^{\si,\de}=Z^{\si,\de}(Z^{\si,\de})^{-1}$
\end{enumerate}
\end{lemma}
\begin{proof}
By Proposition \ref{reduction to semisimple}, we can extend $\si$
 and $\de$ to the classical semisimple quotient ring $Q=Q(R)$ of $R$.
Let $Q= \bigoplus_{i=1}^s B_i = \bigoplus_{j=1}^kA_j$, where
$A_j=\bigoplus _{i\in {\cal O}_j}B_i$, be
 a decompositions of $Q$ described in Lemma \ref{Cauchon and Robson}.
 Then, by the same lemma,  we have  $\T{Q}=\bigoplus_{j=1}^k A_j[x_j;\si_j,\de_j]$,
 where $\si_j=\si|_{A_j}$ and $\de_j=\de|_{A_j}$.

(1).  Since $Q=RZ^{-1}$, in order to show that
$Q=R(Z^{\si,\de})^{-1}$, it is enough to prove that any regular
element $z$ from the center of $R$ is invertible in
$R(Z^{\si,\de})^{-1}$.

Let $z\in Z$ be regular in $Z$. By Corollary \ref{derivation},
$\si|_{Z}$ is an automorphism of $Z$ of finite order $n\ge 1$.
Then, the element $w=z\si(z)\cdots \si^{n-1}(z)\in Z$ is regular
and $\si(w)=w$. From this  we easily deduce  that $z$ is
invertible in $R(Z^{\si})^{-1}$.  This means that
$Q=R(Z^{\si})^{-1}$. Therefore,  we may assume that our regular
element $z$ belongs to $Z^{\si}$.

Recall that $Q=\bigoplus_{j=1}^k A_j$, where $\si(A_j)\subseteq
A_j$ and  $\de(A_j)\subseteq A_j$, for all $j$. Thus, in
particular, $Z(Q)^\si=\bigoplus_{j=1}^k Z(A_j)^{\si_j}$ and we can
present our element $z$ in the form  $z=z_1+\cdots +z_k $, where
$z_j\in Z(A_j)^{\si_j}$, for $1\leq j\leq k$.

Notice that  Lemma \ref{Cauchon and Robson}(3) (when $A_j$ is not
simple) and Theorem  \ref{prime case} (when $A_j$ is a simple
ring) imply that  the restriction $\de|_{A_j}=\de_j$ is always an
inner $\si_j$-derivation of $A_j$  but the case  $A_j$ is a simple
ring of characteristic $p_j\ne 0$ and
$\si_j|_{Z(A_j)}=\mbox{id}|_{Z(A_j)}$. In the later case
 $\de_j(z_j^{p_j})=p_jz_j^{p_j-1}\de_j(z_j)=0$, as $z_j\in Z(A_j)=Z(A_j)^{\si_j}$. When
$\de_j$ is an inner $\si_j$-derivation, then $\de_j(z_j)=0$, as
$z_j\in Z(A_j)^{\si_j}$. In this case we set $p_j=1$. Then
$\de_j(z_j^{p_j})=0$, for any $1\leq j\leq k$. Let
$m=\prod_{j=1}^kp_j$. Then, putting together the above
information, we get $\de(z^m)=\de_1(z_1^m)+\ldots +
\de_k(z_k^m)=0$. This shows that the regular element  $z^m $
belongs to $ Z^{\si,\de}$ and proves that $z$ is  invertible in
$R(Z^{\si,\de})^{-1}$, i.e. $Q$ is a localization of $R$ with
respect to regular elements from $Z^{\si,\de}$.

(2). Let $1=e_1+\ldots + e_k$ be a decomposition of $1\in Q$ into
a sum of central primitive idempotents (i.e. $B_i=e_iR$, for all
$i$).

Let us fix an element $a\in Z^{\si,\de}$   which is regular in
$Z^{\si,\de}$ and $b\in Z$ such that $ab=0$.

Assume that  $b\ne 0$.  Then, there exists an index $s$ such that
$ae_s=0$. Eventually changing the numeration, we may assume that
$s=1$ and $\mathcal{A}=\{e_1,\ldots, e_l\}$ is the orbit of $e_1$
under the action of $\si$ on the set $\{e_i\mid 1\leq i\leq k\}$.
Since $\si(a)=a$, we have $ac=0$, where
$c=\sum_{e_i\in\mathcal{A}}e_i$. Observe  that the element $c$ is
$(\si,\de)$-invariant.

By the statement $(1)$, $Q=R(Z^{\si,\de})^{-1}$. Therefore,  there
exist an element $z\in Z^{\si,\de}$ regular in $R$ and $0\ne u\in
R$ such that $c=uz^{-1}$. Using the fact that the elements $c$ and
$z$ are central $(\si,\de)$-invariant and $z$ is regular in $R$,
one can check that $u\in Z^{\si,\de}$. Since $a$ is regular in $
Z^{\si,\de}$ we obtain  $0\ne au=acz=0$ This contradiction shows
that $b=0$ and this completes the proof of $(2)$.

The statement (3) is a direct consequence of the fact that
$Z(Q)=ZZ^{-1}$ and statements $(1)$ and $(2)$.
\end{proof}

Now we are in position to prove the main theorem of this section:
\begin{theorem}
\label{main semiprime} Suppose that  $R$ is a semiprime \mbox{\rm
PI} ring with the \mbox{\rm ACC} on annihilators and $\si$ is an
injective endomorphism of $R$. Let $Z$ denote the center of $R$
and $Q=Q(R)$. The following conditions are equivalent:
\begin{enumerate}
  \item $\T{R}$ is a \mbox{\rm PI} ring.
    \item $\udim{Z^{\si,\de}}{Z}=\udim{Z(Q)^{\si,\de}}{Z(Q)}$ is finite.
    \item The center of $\Tn$ contains a nonconstant polynomial with
  a regular leading coefficient.
   \end{enumerate}
   If one of the above equivalent conditions holds, then every regular
   element from $Z^{\si,\de}$ is regular in $R$,
   $Q=R(Z^{\si,\de})^{-1}$ and
   $Z(Q)^{\si,\de} = Z^{\si,\de}(Z^{\si,\de})^{-1}$.
\end{theorem}
\begin{proof}
$(1)\Rightarrow (2)$. Suppose (1) holds. Then, by  Proposition
\ref{reduction to semisimple}, $\T{Q}$ is a PI ring  and Corollary
\ref{C semisimple} shows that $\udim{Z(Q)^{\si,\de}}{Z(Q)}<
\infty$.  The equality
$\udim{Z^{\si,\de}}{Z}=\udim{Z(Q)^{\si,\de}}{Z(Q)}$ is a direct
consequence of Lemmas \ref{le} and \ref{udim and localizations}.

$(2)\Rightarrow (3)$. Suppose (2) holds. Then, by Corollary \ref{C
semisimple} applied to the ring $Q$, the ring $\T{Q}$ satisfies a
polynomial identity and  there exists a nonconstant polynomial $f$
in the center of $\T{Q}$ with invertible leading coefficient. In
particular, $\Tn$ is also a PI ring and Lemma \ref{le} shows that
 $Q=R(Z^{\si,\de})^{-1}$. Hence,  there exists an  element
$z\in Z^{\si,\de}$ regular in $R$, such that $zf\in \Tn$. Clearly
$zf$ is central in $\Tn$ and has a regular leading coefficient.

$(3)\Rightarrow (1)$. Notice that $Z(\Tn)\subseteq Z(\T{Q})$ and
every regular element in $R$ is invertible in $Q$. Thus, the
statement (3) together with Corollary \ref{C semisimple} show that
the ring $\T{Q}$ satisfies a polynomial identity. This gives the
thesis.

The above shows that conditions $(1)\div(3)$ are equivalent. The
remaining statements from the theorem are direct consequences of
Lemma \ref{le}.
\end{proof}

Let us remark that the assumption in the above theorem, that the
ring $R$ satisfies the ACC condition on annihilators, is
essential.
\begin{example}
Let $R=\prod_{i=1}^{\infty} \mathbb{C}$, where $\mathbb{C}$
denotes the field of complex numbers, and let $\si$ be the
automorphism of $R$ which is the complex conjugation on every
component $\mathbb{C}$ of $R$. Then $R[x;\si]$ is a PI ring,
$R^\si=\prod_{i=1}^{\infty} \mathbb{R}$ and $\udim{R^\si}{R}$ is
infinite.
\end{example}

\section{\sc \normalsize  Noetherian Coefficient Ring}
Throughout this section $\si$ stands for an arbitrary, not
necessarily injective, endomorphism of a ring $R$.
$\mathcal{B}(R)$  denotes  the prime radical of $R$.

 The following result is due to Mushrub (Cf.\cite{M1}):
 \begin{lemma}\label{Mushrub}
  Suppose that $R$ is a noetherian ring and $\ker\si\subseteq \mathcal{B}(R)$, then $\si(\mathcal{B}(R))\subseteq
  \mathcal{B}(R)$.
 \end{lemma}
 \begin{prop}\label{injective modulo radical}
  Suppose  $R$ is noetherian and $\ker \si\subseteq \mathcal{B}(R)=B$.
  Then $\si^{-1}(B)=B$ and $\si$ induces an injective endomorphism
  $\bar{\si}$  of $R/B$.
 \end{prop}
 \begin{proof} By Lemma \ref{Mushrub}, $\si(B)\subseteq B$. Thus
 $\si$ induces an endomorphism $\bar{\si}$ of $R/B$.

  By assumption, $\ker \si$ is a nilpotent ideal of $R$. Then, it is easy
  to check that also $\si^{-1}(I)$ is a nilpotent ideal, for any nilpotent ideal  $I$
   of $R$. Thus, in particular, $\si^{-1}(B)\subseteq  B$.
  Then $B\subseteq \si^{-1}(\si(B))\subseteq \si^{-1}(B)\subseteq
  B$ and $\si^{-1}(B)=B$ follows. This, in turn, implies that
  $\bar{\si}$ is injective.
 \end{proof}

\begin{prop}\label{ker contained in B}
 Suppose that $\si$ is an endomorphism of a   noetherian \mbox{\rm PI} ring $R$
 such that $\ker \si\subseteq B=\mathcal{B}(R)$. Then $\A{R}$ is a \mbox{\rm PI} ring
 if and only if the restriction of $\bar{\si}$ to   the center $Z$ of $R/B$ is an
 automorphism of $Z$ of finite order.
\end{prop}
\begin{proof}
 By Proposition \ref{injective modulo radical}, $\si$ induces an
 injective endomorphism $\bar{\si}$ of $R/B$. Since $\si(B)\subseteq
 B$ and $B$ is nilpotent, $\A{B}$ is a nilpotent ideal of $\A{R}$.
 Therefore, as  $(R/B)[x;\bar{\si}]$ is isomorphic to $\A{R}/\A{B}$,
 $\A{R}$ is a \mbox{\rm PI} ring iff $(R/B)[x;\bar{\si}]$ is  \mbox{\rm PI}. Now the thesis is a
 direct consequence of Proposition \ref{semiprime endomorhism}.
\end{proof}

If $\si$ is an automorphism of $R$, then
$\si(\mathcal{B}(R))=\mathcal{B}(R)$ and the above proposition is
exactly Corollary 10\cite{DS}, in this case.

 Notice that, by Lemma \ref{known facts}(2),  the ring $(R/B)[x;\bar{\si}]$ from the proof of the
 above theorem is semiprime. Therefore we have:
 \begin{remark}
Suppose that $\si$ is an endomorphism of a   noetherian \mbox{\rm
PI} ring $R$ such that $\ker \si\subseteq \mathcal{B}(R)$. Then
$B(\A{R})=\A{\mathcal{B}(R)}$.
 \end{remark}
It is known that if the ring $R$ is not noetherian, then
$B(\A{R})$ is not necessarily an extension of an ideal of the
coefficient ring $R$ even when $\si$ is an automorphism of $R$.
However, it was observed by Pascaud and Valette in \cite{PV} that
$B(\A{R})=\A{\mathcal{B}(R)}$, provided $\si$ is an automorphism
of $R$ and the Ore extension $\A{R}$ satisfies a polynomial
identity.

The following lemma is crucial in considering the case when $\ker
\si $ is not included in the radical $\mathcal{B}(R)$ of $R$.
\begin{lemma}\label{lifting PI}
 Let $I$ be an ideal of $R$ such that $\si(I)\subseteq I$. Suppose
 that
 $R$ and $\A{(R/I)}$ are \mbox{\rm PI} rings. Then $T=\A{R}/(\A{I}x)$ is also a
 \mbox{\rm PI} ring.
\end{lemma}
\begin{proof} Since $\si(I)\subseteq I$, $\si$ induces an
endomorphism, also denoted by $\si$, on the factor ring $R/I$.

 Notice that  $R\cap \A{I}x=0$, so we can consider $R$ as a subring of $T$.
 Then $I$ is also an ideal of $T$ and $T/I$ is isomorphic to
 $\A{(R/I)}$.

 Let $(x)$ denote the ideal of $T$ generated by the natural image
 of $x$ in $T$. Then $T/(x)\simeq R$. Therefore, as $I\cap (x)=0$,
 there exists an embedding of $T$ into $R\oplus \A{(R/I)}$ which, by
 assumption,  is a \mbox{\rm PI} ring.
\end{proof}
Notice that an endomorphism $\si$ of $R$ induces an endomorphism
of the factor ring $R/\ker\si$. This endomorphism will also be
denoted by $\si$. The above lemma gives us immediately:
\begin{corollary}\label{reduction by kernel}
 Suppose that $R$ is a \mbox{\rm PI} ring. The following conditions are
 equivalent:
\begin{enumerate}
  \item $\A{R}$ is a \mbox{\rm PI} ring;
  \item for any $n\geq 0$, $\A{(R/\ker\si^n)}$ is a \mbox{\rm PI} ring;
  \item there exists $n\geq 0$ such that $\A{(R/\ker\si^n)}$ is a \mbox{\rm PI} ring.
\end{enumerate}

\end{corollary}
\begin{proof}
 The implications $(1)\Rightarrow (2)\Rightarrow (3)$ are clear.

 $(3)\Rightarrow (1)$. Suppose $\A{(R/\ker\si^n)}$ is \mbox{\rm PI} for some $n\geq
 0$. We may assume $n\geq 1$. Let $I=\ker\si^n$. Then, by Lemma \ref{lifting PI}, the ring
 $T=\A{R}/(\A{I}x)$ is \mbox{\rm PI}. Since $I=\ker \si^n$, $x^nI=0$.
 Therefore $(\A{I}x)^{n+1}=0$. This implies that $\A{R}$ satisfies
 a polynomial identity,
\end{proof}

 When $\si$ is an endomorphism, then $\{\ker \si^k\}_{k\geq 1}$
is an increasing sequence of ideals of $R$. Thus, when $R$ is
noetherian, there is $n\geq 1$ such that $\ker\si^n=\ker\si^m$ for
any $m\geq n$ and $\si$ induces an injective endomorphism, also
denoted by $\si$, of the factor ring $R/\ker\si^n$.
\begin{theorem}\label{general case}
 Suppose $R$ is a noetherian \mbox{\rm PI} ring. Let $n\geq 1$
  be such that $\ker \si^n=\ker\si^{n+1}$ and $R'=R/I$, where $I=\ker\si^n$.
  The following conditions are equivalent:
\begin{enumerate}
  \item  $\A{R}$ is a \mbox{\rm PI} ring;
  \item $\A{R'}$ is a \mbox{\rm PI} ring;
  \item  $\bar{\si}$ is an automorphism of finite order on the
  center of $R'/\mathcal{B}(R')$.
\end{enumerate}
\end{theorem}
\begin{proof}
 As we have seen in   comments before the theorem, $\si$ induces an
 injective endomorphism of $R'$. Thus,   Proposition \ref{ker contained in
 B}, shows that the statements (2) and (3) are equivalent.

 The equivalence $(1)\Leftrightarrow (2)$ is given by Corollary
 \ref{reduction by kernel}.
\end{proof}

We have seen in Proposition \ref{injective modulo radical} that if
$\ker\si\subseteq \mathcal{B}(R)$ in a noetherian ring $R$, then
$\si^{-1}(\mathcal{B}(R))=\mathcal{B}(R)$. Hence
$\ker\si^n\subseteq \mathcal{B}(R)$, for any $n\geq 1$. This means
that $R'/\mathcal{B}(R')=R/\mathcal{B}(R)$ in the theorem above,
i.e. Proposition \ref{ker contained in B} and Theorem \ref{general
case} coincide when $\ker\si\subseteq \mathcal{B}(R)$.

 Notice that, as the following standard
 example shows, the ring $\A{R}$ is not necessary noetherian even
 when
 it is \mbox{\rm PI} and $R$ is noetherian.
 \begin{example}
 Let   $K$ be a field and $\si$   a $K$-linear endomorphism of the polynomial ring
 $R=K[y_1,\ldots,y_n]$ given by $\si(y_1)=0$ and $\si(y_k)=y_{k-1}$ for $k>1$.
  It is easy to see that $\A{R}$ is
 neither   left nor right noetherian. By Theorem \ref{general case},
 $\A{R}$ satisfies a polynomial identity. In fact, since
 $\ker \si^n=(y_1,\ldots,y_n)=I$, $\A{(R/I)}=K[x]$ and the proof of
  Theorem \ref{general case} shows that
 $\A{R}$ satisfies the identity
 $[x_1,x_2]^{n+1}=0$.
 \end{example}

For any ring $R$,   $\si$ induces an injective endomorphism of
$R/I$ where $I=\sum_{k=1}^{\infty}\ker\si^{k}$. Thus one could
hope that an analog of Theorem \ref{general case} could hold at
least in the case $I$ is a prime ideal of $R$ (then, by
Proposition \ref{prime PI endoporphism}, statements (2) and (3)
are equivalent and (1) always implies (2)). However this is not
the case as the following example shows.
 \begin{example}
Let   $K$ be a field and $\si$   a $K$-linear endomorphism of the
polynomial ring
 $R=K[y_i\mid i\geq
 1]$ given by $\si(y_1)=0$ and $\si(y_k)=y_{k-1}$ for $k>1$.
 Then $I=\sum_{k=1}^{\infty}\ker\si^{k}=(y_1,y_2,\ldots)$
 and $\A{R}/\A{I}\simeq \A{R/I}\simeq K[x]$ is a \mbox{\rm PI} ring.

 We claim that $\A{R}$ does not satisfy a polynomial identity
 by showing that, for any  $m\geq 2$
 and $k\geq 1$, $\A{R}$ does not satisfy
  the identity  $S_m(x_1,\ldots,x_m)^k$,  where
   $S_m(x_1,\ldots,x_m)$ denotes the standard identity in
 indeterminates $x_1,\ldots,x_m$. To this end, let us fix $n=n(m,k)$
 such that $n>mk$. Then $S=S_m(y_nt,\ldots,y_{n+m-1}t)=(y_n^m+f)t^m$
 for some suitable $f\in \A{R}$ such that $\deg_{y_n}(f)<m$. The
 choice of $n$ implies that $y_n^m+f\not \in \ker\si^{(k-1)m}$.
 Hence $S^k\ne 0$ follows, as $R$ is a domain.
\end{example}

\end{document}